\documentclass{imamci}

\jno{dnnxxx}

\usepackage{natbib}
\usepackage{epsfig} 
\usepackage{mathptmx} 
\usepackage{times} 
\usepackage{amsmath} 
\usepackage{amssymb}  
\usepackage{amsthm}

\usepackage{bm,mathrsfs}

\renewenvironment{proof}[1][Proof]{\noindent\textit{#1. } }{\hfill$\square$}

 \newtheoremstyle{theorem}{6pt}{6pt}{\rm}{}{\sffamily}{ }{ }{}
 \theoremstyle{theorem}
\newtheorem{theorem}{\sc Theorem}[section]

 \newtheoremstyle{algorithm}{6pt}{6pt}{\rm}{}{\sffamily}{ }{ }{}
 \theoremstyle{algorithm}

 \newtheoremstyle{lemma}{6pt}{6pt}{\rm}{}{\sffamily}{ }{ }{}
 \theoremstyle{lemma}
 \newtheorem{lemma}{\sc Lemma}[section]

\newtheoremstyle{case}{6pt}{6pt}{\rm}{}{\sffamily}{. }{ }{}
 \theoremstyle{case}

 \newtheoremstyle{statement}{6pt}{6pt}{\rm}{}{\sffamily}{ }{ }{}
\theoremstyle{statement}

 \newtheoremstyle{corollary}{6pt}{6pt}{\rm}{}{\sffamily}{ }{ }{}
 \theoremstyle{corollary}

  \newtheoremstyle{definition}{6pt}{6pt}{\rm}{}{\sffamily}{ }{ }{}
 \theoremstyle{definition}
 \newtheorem{definition}{\sc Definition}[section]

\newtheoremstyle{example}{6pt}{6pt}{\rm}{}{\sffamily}{ }{ }{}
\theoremstyle{example}

\newtheoremstyle{remark}{6pt}{6pt}{\rm}{}{\sffamily}{ }{ }{}
\theoremstyle{remark}
\newtheorem{remark}{\sc Remark}[section]

\newtheoremstyle{approximation}{6pt}{6pt}{\rm}{}{\sffamily}{ }{ }{}
\theoremstyle{approximation}

\newtheoremstyle{scheme}{6pt}{6pt}{\rm}{}{\sffamily}{ }{ }{}
\theoremstyle{scheme}

\newtheoremstyle{Algorithm}{6pt}{6pt}{\rm}{}{\sffamily}{ }{ }{}
\theoremstyle{Algorithm}

\newtheoremstyle{Assumption}{6pt}{6pt}{\rm}{}{\sffamily}{ }{ }{}
\theoremstyle{Assumption}

\newtheoremstyle{proposition}{6pt}{6pt}{\rm}{}{\sffamily}{ }{ }{}
\theoremstyle{proposition}
\newtheorem{proposition}{\sc Proposition}[section]

\newtheoremstyle{hypo}{6pt}{6pt}{\rm}{}{\sffamily}{ }{ }{}
 \theoremstyle{hypo}

  \newtheoremstyle{Step}{6pt}{6pt}{\rm}{}{}{ }{ }{}
 \theoremstyle{Step}

\numberwithin{equation}{section}

\usepackage{hyperref}   
\hypersetup{
    colorlinks=true,
    linkcolor=blue,
    filecolor=magenta,      
    urlcolor=blue,
        citecolor=red,
    pdftitle={Sharelatex Example},
    pdfpagemode=FullScreen,
}

\usepackage{url}

\usepackage{calc}
\usepackage{pdfsync}
\usepackage{enumerate}

\usepackage{amsmath,amssymb,amsfonts,amsthm}%
\usepackage{color}%
\usepackage{enumitem}%
\usepackage{hyperref}%

\newtheorem{assumption}{Assumption} 
\newcommand{\tm}{\times}%
\newcommand{\Uc}{\mathcal{U}}%
\newcommand{\R}{\mathbb{R}}%
\newcommand{\ccat}[3]{{#1\, \underset{#3}{\lozenge}\,{#2}}}%
\newcommand{\K}{\mathcal{K}}%
\newcommand{\Kinf}{\mathcal{K_\infty}}%
\newcommand{\KL}{\mathcal{KL}}%
\newcommand{\LL}{\mathcal{L}}%
\newcommand{\ep}{\varepsilon}%
\newcommand{\PC}{\mathrm{PC}}%
\newcommand{\UGS}{\mathrm{UGS}}%
\newcommand{\Z}{\mathbb{Z}}%
\newcommand{\id}{\mathrm{id}}%
\newcommand{\dist}{\mathrm{dist}}%

\newcommand{\T}{\ensuremath{\mathcal{T}}}  

\newcounter{syscounter}
\newenvironment{sysnum}{\begin{list}{($\Sigma{\arabic{syscounter}}$)}%
{\settowidth{\labelwidth}{($\Sigma4$)}
\settowidth{\leftmargin}{($\Sigma4$)~}%
\usecounter{syscounter}}}
{\end{list}}

\newcommand \qrq   {\quad\Rightarrow\quad}

\newcommand \sws   {\ \ \ \wedge\ \ \ }

\newcommand \srs   {\ \ \Rightarrow\ \ }

\newcommand \eps {\varepsilon}

\usepackage{csquotes}  
\newcommand\q{\enquote}

\begin{document}

\title{Non-uniform ISS small-gain theorem for infinite networks}
\author{ {\sc Andrii Mironchenko}\\[2pt]
Faculty of Computer Science and Mathematics, University of Passau, Innstra\ss e 33, 94032 Passau, Germany\\[6pt]
{\rm [Received on XXX]}\vspace*{6pt}}
\pagestyle{headings}
\markboth{A. MIRONCHENKO}{\rm Non-uniform ISS small-gain theorem for infinite networks}
\maketitle


\begin{abstract}
{We introduce the concept of non-uniform input-to-state stability for networks.
It combines the uniform global stability with the uniform attractivity of any subnetwork while it allows for non-uniform convergence of all components.
For an infinite network consisting of input-to-state stable subsystems, that do not necessarily have a uniform $\KL$-bound on the transient behavior, we show: If the gain operator satisfies the uniform small-gain condition, then the whole network is non-uniformly input-to-state stable and all its finite subnetworks are input-to-state stable.}
{input-to-state stability, small-gain theorem, stability of networks, nonlinear systems, infinite-dimensional systems.}
\end{abstract}

\section{Introduction}

We are living in the world of networks, which grow steadily in size and in the number of couplings between individual agents.
Smart grids, connected vehicles, swarm robotics, and smart cities are particular examples of such networks, in which the participating agents may be plugged in and out from the network at any time.
Natural generalizations of such large-scale networks are infinite networks, which overapproximate and capture the essence of the original interconnections.
The complexity of such networks motivates to use the bottom-up approach and to establish stability for a complex system based on properties of its less complex components.

During the last 2 decades, vast literature appeared, devoted to spatially invariant systems and/or linear systems consisting of an infinite number of finite-dimensional components, interconnected with each other through the same pattern \cite{BPD02,BaV05,BeJ17,CIZ09,JoB05b}, etc.
Recently, several papers appeared, devoted to the development of methods and tools to analyze and control 
\emph{infinite networks composed of nonlinear infinite-dimensional systems of different nature, which are not necessarily spatially invariant}.
These works are based on the nonlinear small-gain methods and infinite-dimensional input-to-state stability (ISS) theory.

ISS theory has been initiated in \cite{Son89} and has quickly become one of the pillars of nonlinear control theory, including robust stabilization, nonlinear observer design, and analysis of large-scale networks, see \cite{ArK01,KKK95,Son08}. For the analysis of coupled systems, the ISS paradigm is especially fruitful in combination with the small-gain approach. 
In this method, the influence of any subsystem on other subsystems of a network is characterized by so-called gain functions. The gain operator constructed from these functions characterizes the interconnection structure of the network. 
The small-gain theorems state that if the gains are small enough (i.e., the gain operator satisfies some sort of small-gain condition), the network is stable.%

Small-gain theorems originated within the input-output theory of linear systems, for an overview see \cite{DeV09}.
The small-gain technique was extended to nonlinear feedback input-output systems in \cite{Hil91,MaH92}.
ISS paradigm allowed to extend the nonlinear small-gain theorems to the couplings of 2 nonlinear state-space systems in \cite{JTP94,JMW96}, and further to finite networks of input-to-state stable finite-dimensional systems \cite{DRW07,DRW10}.

Substantial progress in the infinite-dimensional ISS theory within the last years~\cite{JNP18,JSZ19,KaK16b,KaK19,MiW18b,TPT17,ZhZ18} (see~\cite{MiP20} for a recent survey) has created the basis which allows extending the small-gain results to finite and infinite networks of infinite-dimensional systems.
Small-gain results for finite networks of evolution equations in Banach spaces, both in trajectory and in Lyapunov formulations, have been developed in \cite{BLJ18,Mir19b,MiI15b} and \cite{KaK18,KaK19b,TWJ12} (see \cite{Mir19b} for more details and references).

\emph{Small-gain analysis of infinite (not necessarily spatially invariant) networks is especially challenging since the gain operator, collecting the information about the internal gains, acts on infinite-dimensional space, in contrast to finite networks of arbitrary nature}. This calls for a careful choice of the infinite-dimensional state space of the overall network and motivates the use of the theory of positive operators on ordered Banach spaces for the small-gain analysis.%

For networks consisting of \emph{exponentially} ISS systems, possessing exponential ISS Lyapunov functions with linear gains, it was shown in \cite{KMS19} that the whole network is exponentially ISS and there is a coercive exponential ISS Lyapunov function for the whole network provided that the spectral radius of the gain operator is less than one. 

Lyapunov-based ISS small-gain theorems have been reported in \cite{DMS19a,DaP20,KMZ21}.
In~\cite{DaP20}, ISS was shown for an infinite network of ISS systems provided that the internal gains capturing the influence of subsystems on each other are all uniformly less than identity, which is a rather conservative condition. In~\cite{DMS19a} the Lyapunov-based small-gain results for infinite networks have been shown, under the assumption of the existence of a linear path of strict decay for the gain operator. Finally, in \cite{KMZ21} the ISS small-gain was shown for infinite networks whose gains are nonlinear functions, which gave a full extension of the Lyapunov-based small-gain theorems for finite networks \cite{DRW10} to the setting of infinite networks. 

Lyapunov-based small-gain theorems require the knowledge of ISS Lyapunov functions for all subsystems. At the same time, even for finite-dimensional systems, it is natural in many applications just to assume that some subsystems are ISS, without requiring the availability of ISS-Lyapunov functions for them, see \cite{JTP94,DRW07}. 
Furthermore, for PDEs with boundary inputs, ISS Lyapunov converse theorems are not available, even for linear analytic systems, and major challenges appear on this way, which have led to the development of alternative non-Lyapunov methods for ISS analysis of infinite-dimensional systems, including admissibility theory \cite{JNP18,JSZ19}, monotonicity-based approach \cite{MKK19,ZhZ20b}, spectral decompositions \cite{KaK16b,KaK19,LhS19}, De Giorgi iteration \cite{ZhZ19b}.
This motivates the development of ISS small-gain theorems, which are not based on the knowledge of ISS Lyapunov functions for subsystems.

A trajectory-based small-gain theorem for infinite networks consisting of nonlinear infinite-dimensional systems has been shown in \cite{MKG20}. It states that \emph{if all the subsystems 
are ISS with a uniform $\KL$-transient bound} and with the gain operator satisfying the so-called monotone limit property, then the network is ISS. 
Furthermore, in \cite{MKG20} it was shown that a network consisting of subsystems that are merely uniformly globally stable (UGS), is again UGS provided that the gain operator satisfies a uniform small-gain condition, which is implied by the monotone limit (MLIM) property.

\subsection*{Contribution}

In this paper, we introduce
  the novel concept of \emph{non-uniform input-to-state stability} for networks, which is equivalent to ISS in the case of finite networks but is weaker than ISS  for infinite networks.
Using this concept we analyze the stability of infinite networks of ISS systems, that, in contrast to \cite{MKG20}, may not have uniform convergence rates of 
individual subsystems, and whose gain operators may not satisfy the MLIM property.

First, we show in Proposition~\ref{prop:Criterion-non-uniform-ISS} that the non-uniform ISS is equivalent to uniform global stability together with a uniform in an initial state but non-uniform in the components of the network version of the asymptotic gain property.

\emph{Our main result is the non-uniform ISS small-gain theorem (Theorem~\ref{thm:nonuniform-ISS_SGT-infinite-interconnections}), showing that an infinite interconnection of ISS systems, that possess a common $\Kinf$-bound (but not necessarily a common $\KL$-bound) on the transient behavior of subsystems, is non-uniformly ISS provided that the gain operator satisfies the uniform small-gain condition.}

We show that under the same requirement any finite subnetwork of the infinite network is ISS, which makes our result applicable to ISS analysis of finite networks of unknown size.

Since for finite networks, the uniform small-gain condition for the gain operator is equivalent to the strong small-gain condition (as shown in \cite{MKG20}), and since for finite networks, non-uniform ISS is equivalent to ISS, our \emph{non-uniform ISS small-gain theorem can be seen as an alternative to \cite{MKG20} way to extend the ISS small-gain theorem from finite to infinite networks}.

Last but not least, our results are valid for a very general class of control systems, including many classes of partial differential equations (PDEs) with distributed and boundary control, time-delay systems, discrete-time systems, etc.

\textbf{Notation.} 
We write $\R$ for the set of real numbers and $\Z$ for the set of integers. $\R_+$ and $\Z_+$ denote the sets of nonnegative reals and integers, respectively.%

We use the following classes of comparison functions:
{\allowdisplaybreaks
\begin{equation*}
\begin{array}{ll}
{\K} &:= \left\{\gamma:\R_+ \to \R_+ \ : \ \gamma\mbox{ is continuous and strictly increasing, }\gamma(0)=0\right\}\\
{\K_{\infty}}&:=\left\{\gamma\in\K \ :\ \gamma\mbox{ is unbounded}\right\}\\
{\LL}&:=\left\{\gamma:\R_+ \to \R_+ \ :\ \gamma\mbox{ is continuous and decreasing with}
 \lim\limits_{t\rightarrow\infty}\gamma(t)=0 \right\}\\
{\KL} &:= \left\{\beta: \R_+^2 \to \R_+ \ : \ \beta(\cdot,t)\in{\K},\ \forall t \geq 0,\  \beta(r,\cdot)\in {\LL},\ \forall r >0\right\}
\end{array}
\end{equation*}
}
For a normed linear space $(W,\|\cdot\|_W)$ and any $r>0$, we write $B_{r,W} :=\{w \in W: \|w\|_W < r\}$ for the open ball of radius $r$ around $0$ in $W$. 
By $\overline{B_{r, W}}$ we denote the corresponding closed ball. 
Throughout the paper, all considered vector spaces are vector spaces over $\R$. 
For a set $U$, we let $U^{\R_+}$ denote the space of all maps from $\R_+$ to $U$.

Given a nonempty index set $I$, we write $\ell_{\infty}(I)$ for the Banach space of all functions $x:I \rightarrow \R$ with $\|x\|_{\ell_{\infty}(I)} := \sup_{i\in I}|x(i)| < \infty$. Moreover, $\ell_{\infty}(I)^+ := \{ x \in \ell_{\infty}(I) : x(i) \geq 0 \mbox{ for all } i \in I \}$. 
If $I = \Z_+$, we simply write $\ell_{\infty}$ and $\ell_{\infty}^+$, respectively. 

By $\wedge$ we denote the logical \q{and}.

\section{Control systems and their properties}
\label{sec:Control systems and their properties}

In this section we introduce the concept of a control system as well as the main stability properties.
\begin{definition}
\label{def:Time-set} 
A \emph{time set} $\T$ is a subgroup of $(\R, +)$.
It induces the \emph{positive time set} $\T_+:=\{t \in \T: t \geq 0\}$.
\end{definition}

By convention, when the time set $\T$ is understood from the context, all
intervals are assumed to be restricted to $\T$. 
For instance,
$[a, b) = \{t \in \T , a \leq t < b\}$ and $t\geq 0$ is used synonymously with $t\in\T_+$.

We define the concept of a (time-invariant) system in the following way (cf. \cite{Wil72,KaJ11} for similar formalisms):
\begin{definition}
\label{Steurungssystem}
Consider the tuple $\Sigma=(\T_+,X,\Uc,\phi)$ consisting of 
\begin{enumerate}[label = (\roman*)]  
		\item A positive time set $\T_+$.
    \item A normed vector space $(X,\|\cdot\|_X)$, called the \emph{state space}, endowed with the norm $\|\cdot\|_X$.
    \item A normed vector \emph{space of inputs} $\Uc \subset \{u:\T_+ \to U\}$ endowed with a norm $\|\cdot\|_{\Uc}$.
		Here $U$ is a certain normed vector space, called the \emph{space of input values}, and 
		$\Uc$ is a linear subspace of $U^{\T_+}$.
		
		We assume that the following two axioms hold:
                    
\emph{The axiom of shift invariance}: for all $u \in \Uc$ and all $\tau\geq0$ the time
shift $u(\cdot + \tau)$ belongs to $\Uc$ with \mbox{$\|u\|_\Uc \geq \|u(\cdot + \tau)\|_\Uc$}.

\emph{The axiom of concatenation}: for all $u_1,u_2 \in \Uc$ and for all $t>0$ the concatenation of $u_1$ and $u_2$ at time $t$, defined by
\begin{equation}
\ccat{u_1}{u_2}{t}(\tau):=
\begin{cases}
u_1(\tau), & \text{ if } \tau \in [0,t], \\ 
u_2(\tau-t),  & \text{ otherwise},
\end{cases}
\label{eq:Composed_Input}
\end{equation}
belongs to $\Uc$.

    \item A map $\phi:D_{\phi} \to X$, $D_{\phi}\subseteq \T_+ \times X \times \Uc$ (called \emph{transition map}), so that for all $(x,u)\in X \tm \Uc$ it holds that $D_{\phi} \cap \big(\T_+ \times \{(x,u)\}\big) = [0,t_m)\tm \{(x,u)\}$, for a certain $t_m=t_m(x,u)\in (0,+\infty]$.
		
		The corresponding interval $[0,t_m)$ is called the \emph{maximal domain of definition} of $t\mapsto \phi(t,x,u)$.
		
\end{enumerate}
The quadruple $\Sigma$ is called a \emph{(control) system}, if the following properties hold:

\begin{sysnum}
    \item\label{axiom:Identity} \emph{The identity property:} for every $(x,u) \in X \times \Uc$
          it holds that $\phi(0, x,u)=x$.
    \item \emph{Causality:} for every $(t,x,u) \in D_\phi$, for every $\tilde{u} \in \Uc$, such that $u(s) =
          \tilde{u}(s)$ for all $s \in [0,t]$ it holds that $[0,t]\tm \{(x,\tilde{u})\} \subset D_\phi$ and $\phi(t,x,u) = \phi(t,x,\tilde{u})$.
					
    \item \label{axiom:Continuity} \emph{Continuity:} for each $(x,u) \in X \tm \Uc$ the map $t \mapsto \phi(t,x,u)$ is continuous on its maximal domain of definition.%
        \item \label{axiom:Cocycle} \emph{The cocycle property:} for all
                  $x \in X$, $u \in \Uc$, for all $t,h \geq 0$ so that $[0,t+h]\tm \{(x,u)\} \subset D_{\phi}$, we have
$\phi(h,\phi(t,x,u),u(t+\cdot))=\phi(t+h,x,u)$.
\end{sysnum}

\end{definition}

\begin{definition}\label{def_forward_completeness} 
We say that a control system $\Sigma = (\T_+,X,\Uc,\phi)$ is \emph{forward complete} if $D_\phi = \T_+ \tm X \tm \Uc$, i.e., $\phi(t,x,u)$ is well-defined for all $(x,u) \in X \tm \Uc$ and $t \geq 0$.
\end{definition}

If $\T=\R$, this class of systems encompasses control systems generated by ordinary differential equations (ODEs), switched systems, time-delay systems, many classes of partial differential equations, important classes of boundary control systems and many other systems. 

For $\T=\Z$, this class includes infinite-dimensional discrete-time systems of the form
\begin{eqnarray}
x(k+1) = A (x(k),u(k)), \quad k\in\Z_+,
\label{eq:Discrete-time-system_nonlinear-with-inputs}
\end{eqnarray}
where $A: X \times U \to X$ is a nonlinear operator and $\Uc$ is the space $\ell_\infty(\Z_+,U)$, defined as a set 
of all $u:\Z_+ \to U$, such that $\|u\|_{\infty} := \sup_{k\in\Z_+}\|u(k)\|_U < \infty$. 
For each initial condition $x \in X$ and each input $u\in\Uc$ the solution of the system \eqref{eq:Discrete-time-system_nonlinear-with-inputs} exists and is unique for all times $k\in\Z_+$. Denoting this solution at time $k$ by $\phi(k,x,u)$, we see that 
\eqref{eq:Discrete-time-system_nonlinear-with-inputs}
gives rise to a forward complete infinite-dimensional discrete-time control system $\Sigma=(\Z, X,\Uc,\phi)$.

For prolongation of solutions, the following property is important (which we adopt from \cite[Ch.~1]{KaJ11b}), which is implied by forward completeness.
\begin{definition}\label{def_BIC} 
We say that a system $\Sigma = (\T_+,X,\Uc,\phi)$ satisfies the \emph{boundedness-implies-continuation (BIC) property} if for each $(x,u)\in X \tm \Uc$ such that the maximal existence time $t_m = t_m(x,u)$ is finite, for any given $M>0$ there exists $t \in [0,t_m)$ with $\|\phi(t,x,u)\|_X > M$, i.e., $\limsup_{t\to t_m-0}\|\phi(t,x,u)\|_X  = \infty$.
\end{definition}

An important concept, introduced for ODE systems in \cite{Son89}, that unifies internal and external stability, is:
\begin{definition}\label{def_ISS}
A system $\Sigma = (\T_+,X,\Uc,\phi)$ is called \emph{(uniformly) input-to-state stable (ISS)} if there exist $\beta \in \KL$ and $\gamma \in \K$ such that for all $(t,x,u) \in D_{\phi}$ the following  holds: 
\begin {equation}\label{eq_iss}
  \|\phi(t,x,u)\|_X \leq \beta(\|x\|_X,t) + \gamma(\|u\|_{\Uc}).%
\end{equation}
\end{definition}

We need also the following property, implied by ISS, and playing an important role for characterization of ISS (cf. \cite{SoW96,MiW18b}):
\begin{definition}
A system $\Sigma = (\T_+,X,\Uc,\phi)$ is called \emph{uniformly globally stable (UGS)} if there exist $\sigma \in\Kinf$ and $\gamma \in \Kinf$ such that for all $(t,x,u) \in D_{\phi}$ it holds that
\begin{equation}\label{eq_UGS}
  \|\phi(t,x,u)\|_X \leq \sigma(\|x\|_X) + \gamma(\|u\|_{\Uc}).%
\end{equation}
\end{definition}

\section{Infinite interconnections}
\label{sec:Infinite interconnections}

Recall the concept of infinite interconnections, indexed by some non-empty set $I$, containing at least 2 elements, as developed in \cite{MKG20} and inspired by \cite[Definition 3.3]{KaJ07}.

For each $i \in I$, let $(X_i,\|\cdot\|_{X_i})$ be the state space of the $i$-th system $\Sigma_i$, which we assume to be a normed vector space.
Before we can specify the space of inputs for $\Sigma_i$, we construct the overall state space.

Let $Q$ be a nonempty subset of $I$. In what follows, we denote by $(x_j)_{j \in Q}$ the vector consisting of $x_j$ for $j\in Q$. 
Define
\begin{equation}
\label{eq:X_Q}
 \hspace{-2mm} X_Q := \big\{ (x_j)_{j \in Q} : x_j \in X_j,\ \forall j \in Q \mbox{ and } \sup_{j \in Q} \|x_j\|_{X_j} < \infty \big\},%
\end{equation}
which becomes a real Banach space with the norm%
\begin{equation*}
  \|x\|_{X_Q} := \sup_{j\in Q}\|x_j\|_{X_j}.%
\end{equation*}
The \emph{state space} for the network we define as $X:=X_I$. 
We identify $X_Q$ with the space 
\begin{eqnarray}
\{ (x_j)_{j \in I} \in X :  x_j=0,\ j \notin Q\},
\label{eq:Identification}
\end{eqnarray}
and thus $X_Q$ is a subspace in $X$ in the sense of this identification. 

Also we use the shorthand notation $X_{\neq i}:=X_{I\backslash\{i\}}$ to denote the \emph{space of internal input values} to the $i$-th subsystem of the network. 

Consider for each $i \in I$ a control system of the form%
\begin{equation}
\label{eq:Sigma-i}
  \Sigma_i = (\T_+, X_i,\PC_b(\T_+,X_{\neq i}) \tm \Uc,\bar{\phi}_i),%
\end{equation}
where $\PC_b(\T_+,X_{\neq i})$, which we call the space of \emph{internal inputs}, is the space of globally bounded piecewise continuous functions, with the norm $\|w\|_{\infty} = \sup_{s \geq 0}\|w(s)\|_X$.

The norm on $\PC_b(\T_+,X_{\neq i}) \tm \Uc$ we define by%
\begin{equation}\label{eq_product_input_norm}
  \|(w,u)\|_{\PC_b(\T_+,X_{\neq i}) \tm \Uc} := \max\left\{ \|w\|_{\infty}, \|u\|_{\Uc} \right\}.%
\end{equation}
Recall that we assume that $\Uc \subset U^{\T_+}$ for some normed space $U$, and $\Uc$ satisfies the axioms of shift invariance and concatenation. Then, by the definition of $\PC_b(\T_+,X_{\neq i})$ and the norm \eqref{eq_product_input_norm}, these axioms are also satisfied for the product space $\PC_b(\T_+,X_{\neq i}) \tm \Uc$.%

According to the definition of $\Sigma_i$, the internal inputs belong to $\PC_b(\T_+,X_{\neq i})$, and thus are globally bounded. 
However, we can extend the definition of $\Sigma_i$ to allow for unbounded inputs, using the \emph{causal extension}: for each $t\geq 0$, each $x_i\in X_i$, each piecewise continuous $\phi_{\neq i}:\T_+\to X_{\neq i}$, and each $u\in\Uc$ we define
\begin{eqnarray}
\bar{\phi}_i(t,x_i,(\phi_{\neq i},u)): = \bar{\phi}_i(t,x_i,(\phi_{\neq i}|_{[0,t]},u)).
\label{eq:Causal-extension}
\end{eqnarray}
Here $\phi_{\neq i}|_{[0,t]}$ is the restriction of 
$\phi_{\neq i}$ to the time-interval $[0,t]$, extended by $0$ to the $(t,+\infty)$.
Note that $\phi_{\neq i}|_{[0,t]}$ is a well-defined (due to the concatenation axiom) and bounded piecewise continuous function, i.e., an admissible input to the system $\Sigma_i$, and hence the term $\bar{\phi}_i(t,x_i,(\phi_{\neq i}|_{[0,t]},u))$ in 
\eqref{eq:Causal-extension} makes sense.

In what follows we understand the systems $\Sigma_i$ in the sense of this causal extension.

\begin{definition}
\label{def_interconnection}
Given the control systems $(\Sigma_i)_{i \in I}$ as above, we call a control system of the form $\Sigma = (\T_+, X,\Uc,\phi)$ the \emph{(feedback) interconnection} of systems $(\Sigma_i)_{i\in I}$ if the following holds:%
\begin{enumerate}
\item[(i)]
The components $\phi_i$ of the transition map $\phi:D_{\phi} \rightarrow X$ satisfy for all $(t,x,u) \in D_{\phi}$
\begin{equation}
\label{eq:Solution-map-for-composite-system}
  \hspace{-8mm}\phi_i(t,x,u) = \bar{\phi}_i(t,x_i,(\phi_{\neq i},u)).
\end{equation}
where $\phi_{\neq i}(\cdot) = (\phi_j(\cdot,x,u))_{j \in I \backslash \{i\}}$.

\item[(ii)] $\Sigma$ has the BIC property.%
\end{enumerate}
We call $\Sigma_i$ the \emph{$i$-th subsystem} of $\Sigma$.%
\end{definition}

To measure the influence of each subsystem at any other subsystem, we define input-to-state stability for subsystems in the following form:
\begin{definition}\label{def_subsys_iss_semimax}
Given the spaces $(X_j,\|\cdot\|_{X_j})$, $j\in I$, and the system $\Sigma_i$ for a fixed $i \in I$, we call $\Sigma_i$  \emph{input-to-state stable (in semimaximum formulation)} if $\Sigma_i$ is forward complete and there are $\gamma_{ij},\gamma_j \in \K \cup \{0\}$ for all $j \in I$ with $\gamma_{ii}=0$ and $\beta_i \in \KL$ such that for all initial states $x_i \in X_i$, all internal inputs $w_{\neq i} = (w_j)_{j\in I \backslash \{i\}} \in \PC_b(\T_+,X_{\neq i})$, all external inputs $u \in \Uc$ and $t \geq 0$:%
\begin{align}
  \|\bar{\phi}_i&(t,x_i,(w_{\neq i},u))\|_{X_i}
	\leq \beta_i(\|x_i\|_{X_i},t) + \sup_{j \in I}\gamma_{ij}(\|w_j\|_{[0,t]}) + \gamma_i(\|u\|_{\Uc}).%
\label{eq:ISS-subsystems}
\end{align}
\end{definition}

The functions $\gamma_{ij}$ are called \emph{internal gains} and $\gamma_i$ is called \emph{external gain}.

Recall that given a nonempty index set $I$, we write $\ell_{\infty}(I)$ for the Banach space of all functions $x:I \rightarrow \R$ with $\|x\|_{\ell_{\infty}(I)} := \sup_{i\in I}|x(i)| < \infty$. A positive cone in $\ell_{\infty}(I)$ we define as
\[
\ell_{\infty}(I)^+ := \{ x \in \ell_{\infty}(I) : x(i) \geq 0 \mbox{ for all } i \in I \}.
\] 

Assuming that all systems $\Sigma_i$, $i\in I$, are ISS, we can define a nonlinear monotone operator $\Gamma_{\otimes}:\ell_{\infty}(I)^+ \rightarrow \ell_{\infty}(I)^+$, called \emph{gain operator}, from the gains $\gamma_{ij}$ as follows:%
\begin{equation}
\label{eq:Gain-operator-semimax}
  \Gamma_{\otimes}(s) := \bigl(\sup_{j\in I}\gamma_{ij}(s_j)\bigr)_{i\in I},\quad s = (s_i)_{i\in I} \in \ell_{\infty}(I)^+.%
\end{equation}

$\Gamma_{\otimes}$ is well-defined provided that the following holds:
\begin{assumption}\label{ass_gammamax_welldef}
For every $r>0$, we have%
\begin{equation*}
  \sup_{(i,j) \in I^2}\gamma_{ij}(r) < \infty.
\end{equation*}
\end{assumption}

For small-gain analysis we need the following property of nonlinear gain operators, considered, e.g., in \cite{DRW07,MKG20}:
\begin{definition}
We say that $\id - \Gamma_{\otimes}$ has the \emph{monotone bounded invertibility (MBI) property} if there exists $\xi \in \Kinf$ such that for all $v,w \in \ell_{\infty}(I)^+$
\begin{equation*}
  (\id - \Gamma_{\otimes})(v) \leq w \quad \Rightarrow \quad \|v\|_{\ell_{\infty}(I)} \leq \xi(\|w\|_{\ell_{\infty}(I)}).%
\end{equation*}
\end{definition}

In \cite{MKG20} it was shown that $\id - \Gamma_{\otimes}$ satisfies the MBI property if and only if the \emph{uniform small-gain condition} holds: There is $\eta \in \Kinf$ such that
\begin{equation}
\label{eq:uSGC-dist-form}
  \hspace{-4mm}\dist(\Gamma_{\otimes}(x) - x,\ell_{\infty}(I)^+) \geq \eta(\|x\|_X), \quad x \in \ell_{\infty}(I)^+.%
\end{equation}
Here $\dist(\Gamma_{\otimes}(x) {-} x,\ell_{\infty}(I)^+) {=} \inf_{y \in \ell_{\infty}(I)^+}\|\Gamma_{\otimes}(x) {-} x{-}y\|_{\ell_{\infty}(I)}$.

If the network is finite (i.e., if $I$ is of finite cardinality), then the uniform small-gain condition for $\Gamma_\otimes$ holds if and only if $\Gamma_\otimes$ satisfies the strong small-gain condition, see \cite[Proposition 14]{MKG20}.

\section{Non-uniform ISS and its characterization}
\label{sec:Non-uniform ISS and its characterizations}

In this section $I$ is a given nonempty index set.
\begin{definition}
\label{def:Non-uniform-ISS} 
Let $\Sigma_i:=(\T_+, X_i,\PC_b(\T_+,X_{\neq i}) \tm \Uc,\bar{\phi}_i)$, $i\in I$ be control systems. Assume that the interconnection $\Sigma=(\T_+, X,\Uc,\phi)$ is well-defined and forward complete.

We call the network $\Sigma$ \emph{non-uniformly input-to-state stable (non-uniformly ISS)}, if there are $(\tilde{\beta}_i)_{i\in I } \subset \KL$, $\tilde{\sigma} \in\Kinf$ and $\gamma\in\Kinf$, such that 
\begin{eqnarray}
\tilde{\beta}_i(r,t)\leq \tilde{\sigma}(r), \quad r\in\R_+,\ t\geq 0,\ i\in I,
\label{eq:Bounds_on_beta_coupled-system}
\end{eqnarray}
and the components $\phi_i(\cdot,x,u)$ of the coupled system for each $i\in I$ satisfy for all $t\geq 0$, $x\in X$ and $u\in\Uc$ the following estimate:
\begin{eqnarray}
\|\phi_i(t,x,u)\|_{X_i} \leq \tilde{\beta}_i(\|x\|_X,t) + \gamma(\|u\|_{\Uc}).
\label{eq:non-uniform-ISS}
\end{eqnarray}
\end{definition}

In this section, we show that if all the systems of the network are ISS and the gain operator satisfies the MBI property, then the network is non-uniformly ISS.
The non-uniformity in Definition~\ref{def:Non-uniform-ISS} amounts to the fact that the convergence rates for components of the network are non-uniform. At the same time, the stability bound $\sigma$ and the asymptotic gain $\gamma$ are the same for all components of the network. 
Furthermore, as we show next, for finite networks the non-uniform ISS is equivalent to ISS, and thus our approach constitutes   an alternative way to generalize the ISS small-gain theorem from finite to infinite networks.%


The non-uniform ISS can be characterized as follows:
\begin{proposition}
\label{prop:Criterion-non-uniform-ISS} 
Let $\Sigma_i:=(\T_+, X_i,\PC_b(\T_+,X_{\neq i}) \tm \Uc,\bar{\phi}_i)$, $i\in I$ be control systems. Assume that the interconnection $\Sigma=(\T_+,X,\Uc,\phi)$ is well-defined.
Then $\Sigma$ is non-uniformly ISS if and only if 
$\Sigma$ is UGS and there is $\hat{\gamma}\in\Kinf$ such that for all $\varepsilon,r>0$ and for each $i\in I$ there is $\tau_i(\ep,r)>0$, such that
\begin{align}
\|x\|_X\leq r &\sws \|u\|_\Uc \leq r \sws t\geq\tau_i(\ep,r)
 \qrq \|\phi_i(t, x, u)\|_{X_i} \leq \ep + \hat{\gamma}(\|u\|_\Uc).
\label{eq:nonuniform-in-i-UAG-bounded-inputs}
\end{align}
\end{proposition}

\begin{proof}
\q{$\Rightarrow$}. Let $\Sigma$ be non-uniformly ISS. Then
\begin{align*}
\|\phi(t,&x,u)\|_{X} =\sup_{i\in I}\|\phi_i(t,x,u)\|_{X_i}
\leq \sup_{i\in I}\tilde{\beta}_i(\|x\|_X,t) + \gamma(\|u\|_{\Uc}) \le \hat{\sigma}(\|x\|_X) + \gamma(\|u\|_{\Uc}),
\end{align*}
which shows UGS of $\Sigma$.

To obtain \eqref{eq:nonuniform-in-i-UAG-bounded-inputs}, take $\hat{\gamma}:=\gamma$, and pick $\tau_i(\ep,r)$ for each $\varepsilon,r>0$ and each $i\in I$ such that $\beta_i(r,\tau_i(\ep,r))\leq \varepsilon$. As $\beta_i\in\KL$, such $\tau_i(\ep,r)$ always exists. 

\q{$\Leftarrow$}. As a well-posed interconnection, $\Sigma$ has BIC property. In combination with UGS, this implies forward-completeness. Furthermore, by UGS of $\Sigma$, there are $\sigma_{UGS},\gamma_{UGS}\in\Kinf$:
\begin{align}
&\hspace{-2mm} i\in I \sws x\in X  \sws u\in \Uc  \sws t\geq0\nonumber\\
& \hspace{-2mm}\srs \|\phi_i(t, x, u)\|_{X_i} \leq \sigma_{UGS}(\|x\|_X) + \gamma_{UGS}(\|u\|_\Uc).
\label{lem7:help}
\end{align}
Pick any $\varepsilon,r>0$. From the above estimate we have for all $i\in I$, $t\geq 0$ 
and any $x\in X$, $u\in \Uc$ with $\|x\|_X \leq r\leq \|u\|_\Uc$ that
\begin{eqnarray}
\|\phi_i(t, x, u)\|_{X_i} \leq \sigma_{UGS}(\|u\|_\Uc) + \gamma_{UGS}(\|u\|_\Uc).
\label{eq:UGS-for-large-u}
\end{eqnarray}
Defining $\gamma(r):=\max\{\sigma_{UGS}(r) + \gamma_{UGS}(r), \hat{\gamma}(r)\}$, for all $r\geq 0$, and combining \eqref{eq:UGS-for-large-u} with 
\eqref{eq:nonuniform-in-i-UAG-bounded-inputs}, we obtain the following convergence estimate for arbitrary inputs (where $\tau_i(\ep,r)$ is as in the assumptions of the lemma):
\begin{align}
i\in I \sws &\|x\|_X\leq r\sws u\in \Uc \sws t\geq\tau_i(\ep,r) 
\qrq \|\phi_i(t, x, u)\|_{X_i} \leq \ep + \gamma(\|u\|_\Uc).
\label{eq:nonuniform-in-i-UAG}
\end{align}

Fix arbitrary $r \in \R_+$ and define $\eps_n:= 2^{-n}  \sigma_{UGS}(r)$, for all $n \in \Z_+$. Due to \eqref{eq:nonuniform-in-i-UAG}, there exists a sequence of times
$\tau_{i,n}:=\tau_i(\eps_n,r)$, $i\in I$, $n\in\Z_+$, which we assume without loss of generality to be strictly increasing in $n$, such that for all $i\in I$, $x \in \overline{B_{r,X}}$, $u \in \Uc$ and $n\in\Z_+$
\[
t \geq \tau_{i,n} \qrq \|\phi_i(t,x,u)\|_{X_i} \leq \eps_n + \gamma(\|u\|_{\Uc}).
\]
From \eqref{lem7:help} we see that we may set $\tau_{i,0} := 0$ for all $i\in I$.
Define $w_i(r,\tau_{i,n}):=\eps_{n-1}$, for $i\in I$, $n \in \Z_+\setminus\{0\}$, and $w_i(r,0):=2\eps_0=2\sigma_{UGS}(r)$.

Now extend the definition of $w_i$ 
to a function
$w_i(r,\cdot) \in \LL$, for any $i\in I$. 
We obtain for $t \in (\tau_{i,n},\tau_{i,n+1})$, $n=0,1,\ldots$ and $x\in B_{r,X}$
that
\[
\|\phi_i(t,x,u)\|_{X_i} \leq \eps_n + \gamma(\|u\|_{\Uc})< w_i(r,t) + \gamma(\|u\|_{\Uc}).
\]
Doing this for all $r \in \R_+$ we obtain the definition of the functions $w_i$, $i\in I$.

Now for each $i\in I$ define $\hat \beta_i(r,t):=\sup_{0 \leq s \leq r}w_i(s,t) \geq
w_i(r,t)$ for $(r,t) \in \R_+ \times \R_+$. From this definition it follows that, 
for each $t\geq 0$, $\hat\beta_i(\cdot,t)$ is 
nondecreasing in the first argument and $\hat\beta_i(r,\cdot)$ is decreasing in the second argument for each $r>0$ as
every $w_i(r,\cdot) \in \LL$.
Moreover, for each fixed $t\geq0$, $\hat \beta_i(r,t) \leq \sup_{0 \leq s \leq r}w_i(s,0)=2\sigma_{UGS}(r)$, which implies that $\hat\beta$ is continuous in the first argument at $r=0$ for any fixed $t\geq0$ and also $\hat \beta_i(0,t)=0$ for any $t\geq 0$.

By \cite[Proposition 9]{MiW19b}, $\hat\beta_i$ can be upper bounded by certain $\tilde{\beta}_i\in \KL$, and 
\eqref{eq:non-uniform-ISS} is satisfied with such $\tilde{\beta}_i$.
Furthermore, there is $\tilde{\sigma}\in\Kinf$, such that \eqref{eq:Bounds_on_beta_coupled-system} holds (choose the same function $\omega$ for all $i\in I$ in the proof of \cite[Proposition 9]{MiW19b}). 
\end{proof}

For finite networks ISS coincides with non-uniform ISS.
\begin{proposition}
\label{prop:ISS-of-finite-networks} 
Let $\Sigma_i:=(\T_+, X_i,\PC_b(\T_+,X_{\neq i}) \tm \Uc,\bar{\phi}_i)$, $i\in I$ be control systems and 
let $I$ be a finite set. Assume that the interconnection $\Sigma=(\T_+, X,\Uc,\phi)$ is well-defined.
Then $\Sigma$ is ISS if and only if $\Sigma$ is non-uniformly ISS.
\end{proposition}

\begin{proof}
As $\Sigma$ has BIC property as a well-posed interconnection, ISS implies forward completeness. 
Clearly, ISS implies non-uniform ISS. The converse follows by setting $\beta(r,t):=\max_{i\in I} \tilde{\beta}_i(r,t)$. As $I$ is a finite set, $\beta \in\KL$, and since $\|\phi(t,x,u)\|_{X}=\max_{i\in I}\|\phi_i(t,x,u)\|_{X_i}$ the estimate \eqref{eq:non-uniform-ISS} implies 
for all $x\in X$,  $u\in \Uc$ and  $t\geq0$
\begin{eqnarray*}
\|\phi(t,x,u)\|_{X} \leq \beta(\|x\|_X,t) + \gamma(\|u\|_{\Uc}),
\end{eqnarray*}
which shows ISS of $\Sigma$.
\end{proof}

We recall a technical lemma, shown in \cite{Mir19b}:
\begin{lemma}\label{lem:LimSupEstimate}
Let $g:\R_+\to\R^p_+$, $p\in\Z_+$ be a globally bounded function and let $f:\R_+\to\R_+$ be an unbounded monotonically increasing function. Then%
\begin{eqnarray}
\lim_{t\to\infty} \sup_{s\geq f(t)} g(s)  =  \lim_{t\to\infty} \sup_{s\geq t} g(s). 
\label{eq:LimSupEstimate}
\end{eqnarray}
\end{lemma}

\section{Non-uniform ISS small-gain theorem}
\label{sec:Non-uniform ISS SGT}

The ISS small-gain theorem in \cite[Theorem 2]{MKG20} states that if all subsystems $(\Sigma_i)_{i\in I}$ are ISS with a uniform transient bound (i.e., there is $\beta\in\KL$: $\beta_i \leq \beta$ for all $i\in I$ pointwise) and with the gain operator satisfying the so-called monotone limit property (which implies MBI property), then the network is ISS.

If the subsystems of the network are ISS but do not have uniform $\KL$-bounds from above for the transient behavior, it is not possible to guarantee ISS of the network, as can be seen on a simple example $\dot{x}_i = -\frac{1}{i}x_i$, $i\in I :=\Z_+$, with the state space $X=\ell_\infty$. Each subsystem is exponentially stable, and thus, ISS. But the overall network is not ISS. In fact, for an initial condition $x_0=(1,1,\ldots)$ it holds for the corresponding solution that $\|\phi(t,x_0)\|_{\ell_\infty} = 1$ for all $t\geq 0$.

In contrast to that, next, we show that if all the subsystems are ISS, have a uniform $\Kinf$-bound (\emph{but not necessarily a uniform $\KL$-bound}) on the transient behavior, and the gain operator satisfies merely MBI property (= uniform small-gain condition), then the network is non-uniformly input-to-state stable.

\begin{theorem}[Non-uniform ISS small-gain theorem]
\label{thm:nonuniform-ISS_SGT-infinite-interconnections} 
Let $\Sigma_i:=(\T_+, X_i,\PC_b(\T_+,X_{\neq i}) \tm \Uc,\bar{\phi}_i)$, $i\in I$ be forward complete control systems, satisfying the ISS estimates as in 
Definition~\ref{def_subsys_iss_semimax}. Let also the interconnection $\Sigma=(\T_+, X,\Uc,\phi)$ be well-defined and the following conditions hold:%
\begin{enumerate}[label=(\roman*)]
	\item\label{itm:nonuniform-ISS-SGT-Ass1} There exist $\gamma \in\K$ and $\sigma \in\K$ such that%
\begin{eqnarray}
\beta_i(r,t)\leq \sigma(r),\  \gamma_i(r) \leq \gamma(r),\  r\in\R_+,\ t\geq 0,\ i\in I.
\label{eq:Bounds_on_beta_and_external_gamma-nonuniform}
\end{eqnarray}
	\item\label{itm:nonuniform-ISS-SGT-Ass2} Assumption \ref{ass_gammamax_welldef} is satisfied for the operator $\Gamma_{\otimes}$ defined via the gains $\gamma_{ij}$ from Definition~\ref{def_interconnection} and $\Gamma_{\otimes}$ has the monotone bounded invertibility property.
	\item\label{itm:nonuniform-ISS-SGT-Ass3} for each $i\in I$ only finitely many elements of $(\gamma_{ij})_{j\in I}$ are nonzero.
\end{enumerate} 
Then $\Sigma$ is non-uniformly ISS.
%
%
\end{theorem}

\begin{proof}
We proceed in two steps.

\noindent\textbf{UGS.} 
As all $\Sigma_i$ are ISS with corresponding $\beta_i$ and gains $\gamma_{ij}$ and $\gamma_i$, $i,j\in I$, in view of 
assumption~\ref{itm:nonuniform-ISS-SGT-Ass1} we have that 
for all initial states $x_i \in X_i$, all internal inputs $w_{\neq i} = (w_j)_{j\in I \backslash \{i\}} \in \PC_b(\T_+,X_{\neq i})$, all external inputs $u \in \Uc$ and $t \geq 0$ we have
\begin{align*}
  \|\bar{\phi}_i&(t,x_i,(w_{\neq i},u))\|_{X_i}
	\leq \sigma(\|x_i\|_{X_i}) + \sup_{j \in I}\gamma_{ij}(\|w_j\|_{[0,t]}) + \gamma_i(\|u\|_{\Uc}).%
\end{align*}
As the gain operator satisfies the MBI property, we obtain forward completeness and UGS of the network $\Sigma$ from the UGS small-gain theorem \cite[Theorem 1]{MKG20} (in this reference only continuous-time systems are treated, but the argument does not change for for general systems considered in this paper).

\noindent\textbf{The estimate \eqref{eq:nonuniform-in-i-UAG-bounded-inputs}.} 
As $\Sigma$ is the interconnection of $(\Sigma_i)_{i\in I}$ and is forward complete, we have $\phi_i(t,x,u) = \bar{\phi}_i(t,x_i,(\phi_{\neq i},u))$ for all $(t,x,u) \in \T_+ \tm X \tm \Uc$ and $i \in I$, with the notation from Definition \ref{def_interconnection}.%

Pick any $r \geq 0$, any $u \in \overline{B_{r,\Uc}}$ and any $x \in \overline{B_{r,X}}$. As $\Sigma$ is UGS, there are $\sigma^{\UGS},\gamma^{\UGS} \in \Kinf$ 
such that%
\begin{equation}
\label{eq:mu-definition}
  \|\phi(t,x,u)\|_X \leq \sigma^{\UGS}(r) + \gamma^{\UGS}(r) =: \mu(r) \quad \forall t \geq 0.
\end{equation}

Given $r\geq 0$, $\ep>0$, and $i\in I$, by ISS of $\Sigma_i$  choose $\tau^*_i = \tau^*_i(\ep,r) \geq 0$ (depending on $i$)  
such that $\beta_{i}(\mu(r),\tau^*_i) \leq \ep$. 

Due to the cocycle property, for all $i \in I$ and $t,\tau \geq 0$ we have%
\begin{align*}
  \phi_i(t + \tau,x,u) &= \bar{\phi}_i(t+\tau,x_i,(\phi_{\neq i},u)) \\
	                     &= \bar{\phi}_i\Big(\tau,\bar{\phi}_i(t,x_i,(\phi_{\neq i},u)),\big(\phi_{\neq i}(\cdot+t),u(\cdot+t)\big)\Big).%
\end{align*}
We obtain for all $x \in \overline{B_{r,X}}$, $u \in \overline{B_{r,\Uc}}$, $\tau \geq \tau^*_i$, $t \geq 0$:
\begin{align}
\label{eq:nuISS_SGT_tmp_estimate_1}
	\|\phi_i(t+\tau,x,u)\|_{X_i} 
	&\leq \beta_i(\|\bar{\phi}_i(t,x_i,(\phi_{\neq i},u))\|_{X_i},\tau) 
																	+ \sup_{j \in I} \gamma_{ij}( \|\phi_j\|_{[t,t+\tau]} ) + \gamma_i(\|u(\cdot+t)\|_{\Uc}) \nonumber\\
																	&\leq \beta_{i}(\|\phi(t,x,u)\|_X,\tau^*_i) + \sup_{j \in I}\gamma_{ij}(\|\phi_j\|_{[t,\infty)}) + \gamma_i(\|u\|_{\Uc}) \nonumber\\
																	&
																	\leq \ep {+} \sup_{j \in I}\gamma_{ij}(\|\phi_j\|_{[t,\infty)}) {+} \gamma_i(\|u\|_{\Uc}).%
\end{align}

Pick any $k \in \Z_+$ and define 
\begin{eqnarray}
B(r,k) := \overline{B_{r,X}} \tm \{ u \in \Uc : \|u\|_{\Uc} \in [2^{-k}r,2^{-k+1}r]\}.
\label{eq:B(r,k)-definition}
\end{eqnarray}

Taking the supremum of \eqref{eq:nuISS_SGT_tmp_estimate_1} over $B(r,k)$, we obtain for all $i \in  I$, $\tau\geq \tau^*_i$:
\begin{align}
\sup_{(x,u) \in B(r,k)}\|\phi_i(t+\tau,x , u)\|_{X_i} 
&\leq \ep + \sup_{j\in I\backslash\{i\}}\gamma_{ij}\big(\sup_{(x,u) \in B(r,k)}\big\|\phi_{j,[t,+\infty)}\big\|_{\infty}\big) + \gamma_i(2^{1-k}r).
\label{eq:nuISS_SGT_tmp_estimate_2}
\end{align}
Thus for all $t \geq 0$ we have
{
\begin{align}
\sup_{s\geq t+\tau^*_i} \sup_{(x,u) \in B(r,k)} \|\phi_i(s, x , u)\|_{X_i} 
&\leq \ep + \sup_{j\in I\backslash\{i\}}\gamma_{ij}\big(\sup_{(x,u) \in B(r,k)} \sup_{s\geq t} \|\phi_j(s, x , u)\|_{X_j}\big) + \gamma_i(2^{1-k}r)\nonumber \\
&= \ep + \hspace{-2mm}\sup_{j\in I\backslash\{i\}}\hspace{-2mm}\gamma_{ij}\big(\sup_{s\geq t} \sup_{(x,u) \in B(r,k)}\hspace{-2mm} \|\phi_j(s, x , u)\|_{X_j}\big) {+} \gamma_i(2^{1-k}r).
\label{eq:nuISS_SGT_tmp_estimate_3}
\end{align}
}

Define%
\begin{align*}
  y_i(r,k) &:= \lim_{t\to +\infty}\sup_{s\geq t}\sup_{(x,u)\in B(r,k)}\|\phi_i(s,x,u)\|_{X_i}
	= \limsup_{t\to +\infty}\sup_{(x,u)\in B(r,k)}\|\phi_i(t,x,u)\|_{X_i}.%
\end{align*}
By Lemma~\ref{lem:LimSupEstimate} it holds that%
\begin{equation*}
  y_i(r,k) = \lim_{t\to +\infty}\sup_{s\geq t +\tau^*_i}\sup_{(x,u)\in B(r,k)}\|\phi_i(s,x,u)\|_{X_i}.%
\end{equation*}
As each $\Sigma_i$ has a finite number of neighbors by assumption~\ref{itm:nonuniform-ISS-SGT-Ass3}, we can take the limit $t\to\infty$ in \eqref{eq:nuISS_SGT_tmp_estimate_3} and obtain
\begin{eqnarray}
\label{eq:nuISS_SGT_tmp_estimate_4}
  y_i(r,k) \leq \ep + \sup_{j\in I}\gamma_{ij}\left(y_j(r,k)\right) + \gamma_i(2^{1-k}r).%
\end{eqnarray}
As \eqref{eq:nuISS_SGT_tmp_estimate_4} is valid for arbitrarily small $\ep>0$, we obtain by computing the limit $\ep\to +0$ that%
\begin{eqnarray}\label{eq:nuISS_SGT_tmp_estimate_5}
  y_i(r,k) \leq \sup_{j\in I\backslash\{i\}}\gamma_{ij}\left(y_j(r,k)\right) + \gamma_i(2^{1-k}r),\quad i\in I.%
\end{eqnarray}
Denote $\vec{\gamma}(u) := ( \gamma_i(\|u\|_{\Uc}) )_{i \in I}$ and $y(r,k):=(y_i(r,k))_{i\in I}$ and note that $y(r,k)  \in\ell_\infty^+(I)$, as the entries $y_i$ are uniformly bounded by $\mu(r)$, and by (i), we have $\vec{\gamma}(u) \in \ell_\infty^+(I)$.

Let us rewrite \eqref{eq:nuISS_SGT_tmp_estimate_5} in a vector form:%
\begin{eqnarray}\label{eq:nuISS_SGT_tmp_estimate_6}
  y(r,k) \leq \Gamma^{ISS}_\otimes\left(y(r,k)\right) + \vec{\gamma}(2^{1-k}r),%
\end{eqnarray}
which we reformulate as:%
\begin{eqnarray*}
  (\mathrm{id}-\Gamma^{ISS}_\otimes)(y(r,k)) \leq \vec{\gamma}(2^{1-k}r).
\end{eqnarray*}
From the assumption \ref{itm:nonuniform-ISS-SGT-Ass2} of the theorem, there is $\xi\in\Kinf$ so that%
\begin{eqnarray}
  \|y(r,k)\|_{\ell_\infty} \leq \xi(\| \vec{\gamma}(2^{1-k}r)\|_{\ell_\infty}) \leq \xi(\gamma(2^{1-k}r)).%
\label{eq:nuISS_SGT_tmp_estimate_7}
\end{eqnarray}

As \eqref{eq:nuISS_SGT_tmp_estimate_7} is the same as $y_i(r,k) \leq \xi(\gamma(2^{1-k}r))$ for all $i\in I$, \eqref{eq:nuISS_SGT_tmp_estimate_7} is equivalent to
existence for any $i\in I$, any $\ep>0$, any $r>0$ and any $k\in \Z_+$ of a time $\tilde{\tau}_i=\tilde{\tau}_i(\ep,r,k)$ such that

\begin{equation} 
\label{eq:nuISS_SGT_intermediate_step}
\begin{split}
\|x\|_X\leq r & \sws  \|u\|_\Uc \in [2^{-k}r,2^{1-k}r] \sws t\geq\tilde{\tau}_i(\ep,r,k)\\
& \qrq \|\phi_i(t, x, u)\|_{X_i}  \leq \ep + \xi(\gamma(2^{1-k}r)).
\end{split}
\end{equation}

Define $k_0=k_0(\ep,r) \in \Z_+$ as the minimal $k$ so that $\xi(\gamma(2^{1-k}r)) \leq \ep$ and let
\[
\hat{\tau}_i(\ep,r):=\max\{\tilde{\tau}_i(\ep,r,k):\ k= 1,\ldots,k_0(\ep,r)\}.
\]
Pick any nonzero $u\in\overline{B_{r,\Uc}}$.
Then there is $k\in \Z_+$ so that
$\|u\|_\Uc \in (2^{-k}r,2^{1-k}r]$.
If $k\leq k_0$ (i.e., if inputs are large enough), then for $t\geq\hat{\tau}_i(\ep,r) $ it holds that 
\begin{eqnarray}
\|\phi_i(t, x, u)\|_{X_i}  \leq \ep + \xi(\gamma(2^{1-k}r))  \leq \ep + \xi(\gamma(2\|u\|_{\Uc})). 
\label{eq:Final_nonuniform-bUAG_implication_1}
\end{eqnarray}
It remains to consider the case when $k>k_0$, i.e., when inputs are small.
The estimate \eqref{eq:nuISS_SGT_intermediate_step} gives convergence time, which depends on $k$ and it is not clear whether the supremum of
$\tilde{\tau}_i(\ep,r,k)$ over all $k\geq k_0$ exists.
To overcome this obstacle and to find the uniform time, we mimic above argument once again, namely: for any $q \in [0,r]$ one can take  
supremum of \eqref{eq:nuISS_SGT_tmp_estimate_1} over $x\in B_r$ and over all $u\in\Uc$: $\|u\|_{\Uc} \leq q$, to obtain for all $i=1,\ldots,n$ and 
all $\tau\geq \tau^*_i(\varepsilon,r) $ that 
\begin{align*} 
\sup_{\|u\|_{\Uc}\leq q}&\sup_{x\in B_r}\|\phi_i(t+\tau,x , u)\|_{X_i}
\leq \ep + \sup_{j\in I\backslash\{i\}}\gamma_{ij}\Big(\sup_{\|u\|_{\Uc} \leq q}\sup_{x\in B_r}\left\|\phi_{j,[t,+\infty)}\right\|_{\infty}\Big) + \gamma_i(q).
\end{align*} 
Defining
\[
z_i(r,q):=  \lim_{t\to +\infty}\sup_{s\geq t}\sup_{\|u\|_{\Uc} \leq q}\sup_{x\in B_r}\|\phi_i(s, x , u)\|_{X_i},
\]
and doing analogous steps as above, we obtain for any $i\in I$, $r>0$ and any $q\leq r$ that
\begin{eqnarray*}
z_i(r,q) \leq \xi(\|\vec{\gamma}(q)\|_{\ell_\infty})
  \leq \xi(\gamma(q)).
\end{eqnarray*}
This means that for any $i\in I$, $\ep>0$, any $r>0$ and any $q\geq 0$ there is a time $\bar{\tau}_i=\bar{\tau}_i(\ep,r,q)$ so that
\begin{align}
\begin{split}
\|x\|_X\leq r & \sws \|u\|_\Uc \leq q \sws t\geq\bar{\tau}_i(\ep,r,q) \\
&\srs \|\phi_i(t, x, u)\|_{X_i}  \leq \ep + \xi(\gamma(q)).
\end{split}
\label{eq:nuISS_SGT_interediate_step_Small_U}
\end{align}
In particular, for $q_0:=2^{1-k_0(\ep,r)}r$ we have  
\begin{equation} 
\label{eq:Final_nonuniformbUAG_implication_2}
\begin{split}
&\|x\|_X\leq r \sws \|u\|_\Uc \leq q_0 \sws t\geq\bar{\tau}_i(\ep,r,q_0) \\
&\ \Rightarrow\ \|\phi_i(t, x, u)\|_{X_i}  \leq \ep + \xi(\gamma(2^{1-k_0(\ep,r)}r))\leq \ep + \ep.
\end{split}
\end{equation}
%
Define 
\[
\tau_i(\ep,r) :=\max\{\hat{\tau}_i(\ep,r), \bar{\tau}_i(\ep,r,q_0) \}.
\]
Combining \eqref{eq:Final_nonuniform-bUAG_implication_1} and \eqref{eq:Final_nonuniformbUAG_implication_2}, we obtain 

\begin{align}
\|x\|_X&\leq r \sws \|u\|_\Uc \leq r \sws t\geq\tau_i(\ep,r) \nonumber\\
&\qrq \|\phi_i(t, x, u)\|_{X_i} \leq 2\ep + \xi(\gamma(2 \|u\|_\Uc)).
\label{eq:non-uniform-bUAG-est-NU-SGT}
\end{align}
where $r\mapsto \xi(\gamma(2 r))$ is a $\Kinf$-function. This shows the estimate \eqref{eq:nonuniform-in-i-UAG-bounded-inputs}.

Finally, Proposition~\ref{prop:Criterion-non-uniform-ISS} proves the claim.
\end{proof}

\begin{remark}
\label{rem:Finitely-many-systems} 
If $I$ is a finite set, then the assumptions \ref{itm:nonuniform-ISS-SGT-Ass1} and \ref{itm:nonuniform-ISS-SGT-Ass3} of 
Theorem~\ref{thm:nonuniform-ISS_SGT-infinite-interconnections} hold. Furthermore, by Proposition~\ref{prop:ISS-of-finite-networks},
non-uniform ISS is equivalent to ISS. And in view of results in \cite{MKG20}, the MBI property for $\id - \Gamma_\otimes$ is equivalent to the assertion that $\Gamma_\otimes$ satisfies the strong small-gain condition.
Hence, we recover from Theorem~\ref{thm:nonuniform-ISS_SGT-infinite-interconnections} the ISS small-gain theorem (in semimaximum formulation) for finite networks of infinite-dimensional systems, shown in \cite{Mir19b}.
\end{remark}

\begin{remark}
\label{rem:Necessity-of-uniform-bounds} 
Uniform stability margin $\tilde{\sigma}$ and a common gain for all components of the network $\gamma$ in the definition of the non-uniform stability can be obtained thanks to the assumption \ref{itm:nonuniform-ISS-SGT-Ass1} in Theorem~\ref{thm:nonuniform-ISS_SGT-infinite-interconnections}. This assumption cannot be dropped, as without it, we cannot even guarantee the existence of solutions in $\ell_\infty$ in general, even if there are no interconnections between subsystems (i.e., $\gamma_{ij}=0$ for all $i,j$).
\end{remark}

\section{Applications to stability analysis of finite networks of an unknown size}
\label{sec:Analysis of finite networks of an unknown size}

We are going to show that \emph{our non-uniform ISS small-gain theorem can be used to analyze (uniform) ISS of finite networks of an unknown size}, which are ubiquitous in many real-world applications.

In this case, we treat $(\Sigma_i)_{i\in I}$ as the family of all \emph{possible} subsystems of which the network may be constructed of, where $\Sigma_i$ is defined as in \eqref{eq:Sigma-i}.

Denote by $Q \subset I$ an index set, that is fixed but usually unknown and that enumerates all components (subsystems) which are a part of the network. We assume that the systems $(\Sigma_i)_{i\in I\backslash Q}$ are not interacting with the systems $(\Sigma_i)_{i\in Q}$, which we express by saying that the inputs from $(\Sigma_i)_{i\in I\backslash Q}$ to $\Sigma_j$ for any $j\in Q$ are zero.

Our definition of interconnection (Definition~\ref{def_interconnection}) has been introduced for the interconnection of all systems in the network, and to use this construction for the coupling of a subset of systems, we restrict the systems $(\Sigma_i)_{i\in Q}$ to smaller input spaces.

Using the definition of $X_Q$ as in \eqref{eq:X_Q}, we define the system $\Sigma_i$ with inputs restricted to systems $(\Sigma_i)_{i\in Q}$:
\begin{equation}
\label{eq:Sigma-i-restricted}
  \tilde{\Sigma}_i = (\T_+, X_i,\PC_b(\T_+,X_{Q\backslash\{i\}}) \tm \Uc,\tilde{\phi}_i),%
\end{equation}
where $\tilde{\phi}_i$ is a restriction of $\bar{\phi}_i$ to 
\[
\{(t,x_i,v)\in D_{\bar{\phi}_i}:\quad v \in \PC_b(\T_+,X_{Q\backslash\{i\}}) \tm \Uc\},
\]
where we again identify $X_{Q\backslash\{i\}}$ with the corresponding space \eqref{eq:Identification}.

Assume that the interconnection $\Sigma_Q$ of the systems $(\tilde{\Sigma}_i)_{i\in Q}$ is well-posed in the sense of Definition~\ref{def_interconnection} and that each system $\Sigma_i$, $i\in I$ is ISS in semimaximum formulation, as in Definition~\ref{def_subsys_iss_semimax}, with corresponding gains $\gamma_{ij}$ and the induced operator $\Gamma_{\otimes}$, defined by \eqref{eq:Gain-operator-semimax}.
The gain operator $\Gamma_{\otimes}$ characterizes the interconnection structure of the \emph{maximal network} $(\Sigma_i)_{i\in I}$.

Let us consider now the \emph{real network} $\Sigma_Q$ of subsystems $(\tilde{\Sigma}_i)_{i\in Q}$.
As the inputs from $(\Sigma_i)_{i\in I\backslash Q}$ to $\Sigma_j$ for any $j\in Q$ are zero, we can assume that 
$\gamma_{ij}\equiv 0$ if either $i\in I\backslash Q$, or $j\in I\backslash Q$.

Hence the subsystems  $(\tilde{\Sigma}_i)_{i\in Q}$ are also ISS in semimaximum formulation, with the gain operator
$\Gamma_{\otimes,Q}:\ell_{\infty}(Q)^+ \rightarrow \ell_{\infty}(Q)^+$ induced by the gains $(\gamma_{ij})_{i,j\in Q}$ as follows:%
\begin{equation}
\label{eq:Gain-operator-semimax-restricted}
  \Gamma_{\otimes,Q}(s) := \bigl(\sup_{j\in Q}\gamma_{ij}(s_j)\bigr)_{i\in Q},\quad s = (s_i)_{i\in Q} \in \ell_{\infty}(Q)^+.%
\end{equation}

We need the following lemma 
\begin{lemma}
\label{lem:Monotone-invertibility-of-smaller-operators} 
Assume that $\id - \Gamma_\otimes$ satisfies the MBI property. Let $B:\ell_\infty^+(I) \rightarrow \ell_\infty^+(I)$ be another operator such that $B(s) \leq \Gamma_\otimes(s)$ whenever $s\in \ell_\infty^+(I)$.
Then $\id - B$ satisfies the MBI property.
\end{lemma}

\begin{proof}
As $B\leq \Gamma_\otimes$, for any $v\in \ell_\infty^+(I)$ it holds that  $(\id - B)(v) = (\id - \Gamma_\otimes)(v) + (\Gamma_\otimes-B)(v)\geq (\id - \Gamma_\otimes)(v)$.
Thus, if $(\id - B)(v) \leq w$, then $(\id - \Gamma_\otimes)(v) \leq w$, and by monotone bounded invertibility property of $\Gamma_\otimes$
it holds that $\|v\|_X \leq \xi(\|w\|_X)$, for some $\xi\in\Kinf$, independent on $w,v\in \ell_\infty^+(I)$.
\end{proof}

We have the following result:
\begin{lemma}
\label{lem:Monotone-invertibility-restricted-gain-operators} 
If $\Gamma_{\otimes}$ satisfies the monotone bounded invertibility property, then $\Gamma_{\otimes,Q}$ satisfies the monotone bounded invertibility property for any $Q \subset I$.
\end{lemma}

\begin{proof}
Let $\tilde{w},\tilde{v}\in \ell_{\infty}(Q)^+$ be such that 
\begin{eqnarray}
(\id - \Gamma_{\otimes,Q})\tilde{w} \leq \tilde{v}.
\label{eq:Premise-MBI-Gamma-Q}
\end{eqnarray}
Define $w = (w_i)_{i\in I}$ and $v = (v_i)_{i\in I} \in \ell_{\infty}(I)^+ $ by 
$w_i = \tilde{w}_i$, $v_i = \tilde{v}_i$, $i\in Q$ and $v_i = w_i = 0$, $i\in I\backslash Q$.

Also define $\Gamma_{\otimes,I}:  \ell_{\infty}(I)^+  \to  \ell_{\infty}(I)^+$ by 
\begin{equation}
\label{eq:Gain-operator-semimax-extending-restriction}
  \Gamma_{\otimes,I}(s) := \bigl(\sup_{j\in I}\tilde{\gamma}_{ij}(s_j)\bigr)_{i\in I},\quad s = (s_i)_{i\in I} \in \ell_{\infty}(I)^+,%
\end{equation}
where $\tilde{\gamma}_{ij}\equiv 0$ if either $i\in I\backslash Q$ or $j\in I\backslash Q$, and $\tilde{\gamma}_{ij} = \gamma_{ij}$ otherwise.

As $\Gamma_{\otimes,I}\leq \Gamma_{\otimes}$, $\Gamma_{\otimes,I}$ has MBI property by Lemma~\ref{lem:Monotone-invertibility-of-smaller-operators}. 

Furthermore, \eqref{eq:Premise-MBI-Gamma-Q} holds if and only if $(\id - \Gamma_{\otimes,I})w \leq v$, and by MBI property of $\Gamma_{\otimes,I}$, it holds that 
$\|w\|_{\ell_\infty(I)} \leq \xi(\|v\|_{\ell_\infty(I)})$. As $\|w\|_{\ell_\infty(I)} = \|\tilde{w}\|_{\ell_\infty(Q)}$ and $\|v\|_{\ell_\infty(I)} = \|\tilde{v}\|_{\ell_\infty(Q)}$,
we obtain that $\|\tilde{w}\|_{\ell_\infty(Q)} \leq \xi(\|\tilde{v}\|_{\ell_\infty(Q)})$ holds, with $\xi$ independent of $\tilde{w}, \tilde{v}$. This shows the claim.
\end{proof}

The following result tells that the non-uniform ISS small-gain theorem can be effectively used for ISS analysis of finite networks of unknown size, even if there is no a priori uniform $\KL$-bound on the transient behavior of subsystems from which the network consists of.


\begin{theorem}[Non-uniform ISS small-gain theorem for subnetworks]
\label{thm:Subnetworks-stability-NU-SGT-based} 
Let $\Sigma_i:=(\T_+,X_i,\PC_b(\T_+,X_{\neq i}) \tm \Uc,\bar{\phi}_i)$, $i\in I$ be forward complete control systems, satisfying the ISS estimates as in Definition~\ref{def_subsys_iss_semimax}. 
Furthermore, let the conditions \ref{itm:nonuniform-ISS-SGT-Ass1}--\ref{itm:nonuniform-ISS-SGT-Ass3} of Theorem~\ref{thm:nonuniform-ISS_SGT-infinite-interconnections} hold.

Then for any subset $Q \subset I$ such that  $\Sigma_Q=(\T_+,X_Q,\Uc,\phi)$ (as defined in this section) is well-posed, $\Sigma_Q$ is non-uniformly ISS.
If $Q$ is a finite set, then $\Sigma_Q$ is ISS.
\end{theorem}

\begin{proof}
Non-uniform ISS of $\Sigma_Q$ follows from Theorem~\ref{thm:nonuniform-ISS_SGT-infinite-interconnections} and Proposition~\ref{lem:Monotone-invertibility-restricted-gain-operators}.
If $Q$ is a finite set, then ISS of $\Sigma_Q$ follows by Proposition~\ref{prop:ISS-of-finite-networks}.
\end{proof}

Assumption \ref{itm:nonuniform-ISS-SGT-Ass3} of Theorem~\ref{thm:nonuniform-ISS_SGT-infinite-interconnections} needs to be satisfied in 
Theorem~\ref{thm:Subnetworks-stability-NU-SGT-based} only in case if $Q$ is of infinite cardinality. For finite set $Q$ this assumption is not needed, as for the subnetwork $Q$ it will be always fulfilled.


\begin{remark}
\label{rem:Uniform-SGT-for-subnetworks} 
Theorem~\ref{thm:Subnetworks-stability-NU-SGT-based} gives a condition for ISS of any finite subnetwork, but the functions $\beta$ and $\gamma$ in the ISS definition may depend on $Q$, i.e., we do not have uniform $\beta$ and $\gamma$ for all finite $Q \subset I$.
To state the ISS Small-gain theorem for subnetworks, which guarantees such a uniformity, we need to require stronger conditions on the gain operator, such as a monotone limit property of $\Gamma_\otimes$, see \cite[Theorem 2]{MKG20}. 
\end{remark}

\section{Funding}

This work was supported by the German Research Foundation (DFG) [MI 1886/2-1 to A.M.].

\bibliography{C:/Users/Andrii/Dropbox/TEX_Data/Mir_LitList_NoMir,C:/Users/Andrii/Dropbox/TEX_Data/MyPublications}

\end{document}